\documentclass[11pt]{article}

\usepackage{amssymb,amsmath,amsthm}
\usepackage{epsfig}
\usepackage{breqn}
\usepackage{footnpag}
\usepackage{rotating}
\usepackage{amsfonts}
\usepackage{setspace}
\usepackage{fullpage}
\usepackage{enumitem}
\usepackage{bbold} 
\usepackage{comment}
\usepackage{pgf,tikz}
\usepackage{mathtools}  
\usepackage{hyperref}
\mathtoolsset{showonlyrefs} %to avoid too many numbers next to equations

\newtheorem{thm}{Theorem}%[section]
\newtheorem{claim}[thm]{Claim}
\newtheorem{lem}[thm]{Lemma}
\newtheorem{cor}[thm]{Corollary}
\newtheorem*{defi}{Definition}

\newtheorem{conj}[thm]{Conjecture}
\newtheorem{obs}[thm]{Observation}

\def\ul{\underline}

\newcommand{\F}{\mathcal{F}}

\newcommand{\Q}{\mathcal{Q}}

\newcommand{\abs}[1]{\left\lvert{#1}\right\rvert}

\newcommand{\bfly}[1]{
\begin{tikzpicture}[x=1cm, y=1cm, scale = .1]
\draw (0,0)-- (0,2);
\draw (2,0)-- (2,2);
\draw (0,0)-- (2,2);
\draw (0,2)-- (2,0);
\fill [color=black] (0,0) circle (10pt);
\draw[color=black] (1,1) node {};
\fill [color=black] (0,2) circle (10pt);
\draw[color=black] (1,1) node {};
\fill [color=black] (2,0) circle (10pt);
\draw[color=black] (1,1) node {};
\fill [color=black] (2,2) circle (10pt);
\draw[color=black] (1,1) node {};
\end{tikzpicture}
}

\def\Times{\!\!\times\!\!.. \!\!\times\!\!}

\begin{document}

\title{Forbidden hypermatrices imply general bounds on\newline induced forbidden subposet problems}
\author{Abhishek Methuku\footnote{ Department of Mathematics, Central European University, Budapest. e-mail: \href{mailto:abhishekmethuku@gmail.com}{abhishekmethuku@gmail.com}} \and D\"om\"ot\"or P\'alv\"olgyi\footnote{Institute of Mathematics, E\"otv\"os University, Budapest.
Research supported by Hungarian National Science Fund (OTKA), grant PD 104386 and the J\'anos Bolyai Research Scholarship of the Hungarian Academy of Sciences. 
e-mail: \href{mailto:dom@cs.elte.hu}{dom@cs.elte.hu}}}
%\date{}
\maketitle

\begin{abstract}
We prove that for every poset $P$, there is a constant $C$ such that the size of any family of subsets of $[n]$ that does not contain $P$ as an induced subposet is at most $C{\binom{n}{\lfloor\frac{n}{2}\rfloor}}$, settling a conjecture of Katona, and Lu and Milans \cite{Milans}.
We obtain this bound by establishing a connection to the theory of forbidden submatrices and then applying a higher dimensional variant of the Marcus-Tardos theorem, proved by Klazar and Marcus. %, for which we give a new proof(Is this clearer?). 
We also give a new proof of their result.
\end{abstract}

\section{Introduction}
We are interested in the largest family of subsets of $[n] := \{1,2,\ldots, n \}$ avoiding certain subposets. 

Our main result is the following.

\begin{thm}
\label{mainthm}
For every poset $P$, there is a constant $C$ such that the size of any family of subsets of $[n]$ that does not contain an induced copy of $P$ is at most $C {\binom{n}{\lfloor\frac{n}{2}\rfloor}}.$
%Moreover, $C=2^{k^{\Theta(d)}}$, where $k$ is the size and $d$ is the dimension of the poset $P$.
\end{thm}

To establish this theorem, we use a method that is a certain generalization of the so-called {\em circle method} of Katona \cite{Circlemethod}.
That is, we double count pairs $(\Q,F)$ that satisfy a certain property where $\Q$ is a family of subsets built from a partition of permutation of $[n]$.
What is novel in our paper is that we associate a $d$-dimensional $0-1$ hypermatrix to $\Q$ and we establish a connection to the theory of forbidden submatrices with the following lemma.

\begin{lem}\label{matrix2poset}
For every $d$-dimensional poset $P$ on $k$ elements, there is a $d$-dimensional $k\Times k$ permutation hypermatrix $M_P$ such that if every $M_P$-free $d$-dimensional $n\Times  n$ size $0-1$ hypermatrix has at most $Kn^{d-1}$ non-zero elements, then
$La^{\#}(n, P) \le 2^dK\binom n{\lfloor \frac n2 \rfloor}$.
\end{lem}

Then we apply the following theorem, which is a higher dimensional variant of the Marcus-Tardos theorem \cite{MarTar} about forbidden submatrices.

\begin{thm}[Klazar-Marcus \cite{KlaMar}]\label{hypermarcustardos} If a $d$-dimensional $n\Times  n$ size $0-1$ hypermatrix does not contain a given $d$-dimensional $k\Times  k$ permutation hypermatrix, then it has at most $O(n^{d-1})$ non-zero elements.
\end{thm}

Notice that Theorem \ref{mainthm} follows from Lemma \ref{matrix2poset} and Theorem \ref{hypermarcustardos}.

The organization of this paper is the following.
In the rest of this section we survey results about forbidden weak and induced posets and forbidden submatrices.
In Section \ref{sec:matrix2poset} we prove our main result, Lemma \ref{matrix2poset} that establishes the connection between the two theories.
In Section \ref{sec:hypermarcustardos} we prove Theorem \ref{hypermarcustardos} to make our paper self-contained, and also because our proof substantially differs from the proof in \cite{KlaMar} (in fact we have learned about their result only after independently proving Theorem \ref{hypermarcustardos}).
Finally, we make some further observations and conjectures in Section \ref{sec:concluding}.

\numberwithin{thm}{section}
\subsection{Forbidden weak subposets}
\begin{defi}
Let $P$ be a finite poset, and $\F$ be a family of subsets of $[n]$.
We say that $P$ is contained in $\F$  as a {\em weak} subposet if and only if there is an injection $\alpha : P \rightarrow \mathcal{F}$ satisfying $x_1 <_p x_2 \Rightarrow \alpha(x_1)\subset \alpha(x_2)$ for all $x_1,x_2\in P$.
$\F$ is called $P$-free if $P$ is not contained in $\F$ as a weak subposet.  
We define the corresponding extremal function as
$$ La(n,P) := \max \{ \abs{\F} ~ \mid~ \mathcal{F} ~\text{is $P$-free}\}.$$      
\end{defi}

We denote the number of elements of a poset $P$ by $|P|$.
The linearly ordered poset on $k$ elements, $a_1 < a_2 < \ldots < a_k$, is called a chain of length $k$, and is denoted by $P_k$. 
Using our notation Sperner's theorem can be stated as follows.

\begin{thm} [Sperner \rm \cite{Sper}]
\begin{equation*}
La(n, P_2) = {\binom {n} {\lfloor n/2 \rfloor}}.
\end{equation*}
\end{thm}

Erd\H{o}s extended Sperner's theorem to $P_k$-free families for all $k \ge 2$. 

\begin{thm} [Erd\H{o}s \rm \cite{Erd}]
\label{Erd}
$La(n, P_{k})$ is equal to the sum of the $k-1$ largest binomial coefficients of order $n$.
This implies
\begin{equation*}
La(n, P_k) \leq (k-1){\binom {n} {\lfloor n/2 \rfloor}}.
\end{equation*}
\end{thm} 

Notice that, since any poset $P$ is a weak subposet of a chain of length $|P|$, Theorem \ref{Erd} implies

\begin{cor}\label{eq:genP}
\begin{equation*}
 La(n,P) \le (|P|-1){\binom {n} {\lfloor n/2 \rfloor}}=O \left( \binom {n} {\lfloor n/2 \rfloor} \right ).
\end{equation*}
\end{cor}

Later Katona, Tarj\'{a}n, De Bonis, Swanepoel, Griggs, Lu, Li, Thanh, Methuku and Tompkins \cite{Krfork,DKS,GriggsK,GLu,KatTar,MCasey,Thanh} studied various other posets including brushes, two-end-forks, $N$, diamond, butterfly, skew butterfly, cycles $C_{4k}$ on two levels.
One of the first general results is due to Bukh who obtained bounds on $La(n, P)$ for all posets whose Hasse diagram is a tree.
Let $h(P)$ denote the height (maximum length of a chain) of $P$. 

\begin{thm} [Bukh \cite{Bukh}] 
\label{Bukh}
If $T$ is a finite poset whose Hasse diagram is a tree of height $h(T) \ge 2$, then
\begin{equation}
La(n,T)= (h(T)-1){\binom {n} {\lfloor n/2 \rfloor}} \left(1+O(1/n)\right).
\end{equation}
\end{thm}

Using general structures instead of chains for double counting,
Burcsi and Nagy obtained a similar but weaker version of this theorem for general posets.
Later this was generalized by Chen and Li for $m>1$.

\begin{thm} [Burcsi-Nagy, Chen-Li \cite{BurcsiNagy,ChenLi}]
\label{generalbound}
For any poset $P$, when $n$ is sufficiently large, the inequality
\begin{equation}
\label{eq:genChenLi}
La(n,P) \le \frac{1}{m+1} \left(\abs{P} + \frac{1}{2}(m^2 +3m-2)(h(P)-1) -1 \right) {\binom {n} {\lfloor n/2 \rfloor}}
\end{equation}
holds for any fixed $m \ge 1$.
\end{thm}

Very recently, this general bound was improved by Gr\'osz, Tompkins and the first author.

\begin{thm}[Gr\'osz-Methuku-Tompkins \cite{GMT}]
For any poset $P$, when $n$ is sufficiently large, the inequality
\begin{equation}
La(n,P) \le  \frac{1}{2^{k-1}} \left(\abs P + (3k-5)2^{k-2}(h(P)-1) - 1 \right) {n \choose \left\lfloor n/2\right\rfloor }
\end{equation}
holds for any fixed $k\geq 2$.
\end{thm}

\subsection{Forbidden induced subposets}
\begin{defi}
We say that $P$ is contained in $\F$ as an {\em induced} subposet if and only if there is an injection $\alpha : P \rightarrow \mathcal{F}$ satisfying $x_1 <_p x_2 \Leftrightarrow \alpha(x_1)\subset \alpha(x_2)$ for all $x_1,x_2\in P$.
$\F$ is called induced $P$-free if $P$ is not contained in $\F$ as an induced subposet. 
We define the corresponding extremal function as
$$ La^{\#}(n,P) := \max \{ \abs{\F} ~ \mid~ \mathcal{F} ~\text{is induced $P$-free}\}.$$ 
\end{defi}

Despite considerable progress made on forbidden weak subposets, very little is known about forbidden induced subposets.
The obvious exception is $P_k$, as here the weak and induced containment are equivalent, thus Theorem \ref{Erd} implies $La(n, P_{k}) = La^{\#}(n,P_k)$.

The first result of this type is due to Carroll and Katona \cite{CarK} who showed that ${\binom {n} {\lfloor n/2 \rfloor}} (1 + \frac{1}{n} + \Omega(\frac{1}{n^2})) \le La^{\#}(n,V_2) \le {\binom {n} {\lfloor n/2 \rfloor}} (1 + \frac{2}{n} + O(\frac{1}{n^2}))$.
Recently Boehnlein and Jiang extended Theorem \ref{Bukh} to induced containment.

\begin{thm} [Boehnlein-Jiang \cite{EdTao}]
If $T$ is a finite poset whose Hasse diagram is a tree of height $h(T) \ge 2$, then
\begin{equation}
La^{\#}(n,T)= (h(T)-1){\binom {n} {\lfloor n/2 \rfloor}} \left(1+o(1)\right).
\end{equation}
\end{thm}

Boehnlein and Jiang have also shown that even though $La(n,P)$ and $La^{\#}(n,P)$ are asymptotically equal for posets whose Hasse diagram is a tree, there are posets $P$ for which their ratio,
$\frac{La^{\#}(n,P)}{La(n,P)}$ can be arbitrarily large. Very recently, Patk\'os \cite{Patkos} determined the asymptotic behavior of $La^{\#}(n,P)$ for some classes of posets, namely, the complete $2$ level poset $K_{r,s}$ and the complete multi-level poset $K_{r,s_1,...,s_j,t}$ when all $s_i$'s either equal $4$ or are large enough and satisfy an extra condition.

Note that while $La(n,P)  \le La^{\#}(n,P)$, no nontrivial upper bound for general $P$, was known for $La^{\#}(n,P)$ when the Hasse diagram of $P$ contains a cycle.
It was conjectured by Katona and, independently, by Lu and Milans that the analogue of Corollary \ref{eq:genP} holds for induced posets as well. 

\begin{conj}[Katona, Lu-Milans \cite{Milans}]
\label{conj}
\begin{equation}
 \label{eq:geninducedP}
La^{\#}(n,P) = O \left( \binom {n} {\lfloor n/2 \rfloor} \right ).
\end{equation}
\end{conj}

The main result of our paper is to prove Conjecture \ref{conj} for all posets $P$.

\medskip
\noindent{\bf Theorem \ref{mainthm}.} {\em For every poset $P$, there is a $C$ such that
\begin{displaymath}
La^{\#}(n, P) \le C {\binom{n}{\lfloor\frac{n}{2}\rfloor}}.
\end{displaymath}
%Moreover, $C=2^{k^{\Theta(d)}}$, where $k$ is the size and $d$ is the dimension of the poset $P$.
}
\medskip

It is interesting to note that the constant $C$ in our upper bound on $La^{\#}(n,P)$ depends on the {\em dimension} of $P$, as opposed to $h(P)$ in the upper bound on $La(n,P)$ in Theorem \ref{Bukh} and \ref{generalbound}. 
This notion is defined as follows.

\begin{defi}\label{posetdim}
The {\em dimension} of a poset $P$ is the least integer $t$ for which there exists $t$ linear orderings, $<_1, <_2 \ldots <_t$, of the elements of $P$ such that for every $x$ and $y$ in $P$, $x <_{P} y$ if and only if $x <_i y$ for all $1 \le i \le t$.
\end{defi}

\subsection{Forbidden submatrices} 
A $d$-dimensional {\em hypermatrix} is an $n_1\Times  n_d$ size ordered array.
For short, we refer to such a hypermatrix as a {\em $d$-matrix} of size $n^d$ if $n_1=\ldots=n_d=n$.
So a vector is a $1$-matrix and a matrix is a $2$-matrix. 
We refer to the entries of a $d$-matrix $M$ as $M(\ul i)$ where
$\ul i=(i_1,\ldots i_d)$ and $1\le i_j\le n_j$ for every $1\le j\le d$.
A {\em $j$-column} is a set of entries $\{M(\ul i)\mid 1\le i_j\le n_j\}$.
In this paper we only deal with $d$-matrices whose entries are all $0$ and $1$.
We denote the number of $1$'s in a $d$-matrix $M$ by $|M|$.

\begin{defi}
We say that a $d$-matrix $M$ {\em contains} a $d$-matrix $A$ if it has a $d$-submatrix $M'\subset M$
that is of the same size as $A$ such that $A(\ul i)=1 \Rightarrow M'(\ul i)=1$.
If $M$ does not contain $A$ then we say that $M$ is {\em $A$-free}.
We define the corresponding extremal function as
$$ex_d(n_1\Times n_d,A) := \max \{ \abs{M} ~ \mid~ M ~\text{is }A\text{-free, }d\text{-dimensional of size }n_1\Times  n_d\}$$
and if $n_1=\ldots=n_d$, we use $ex_d(n,A) := ex_d(n_1\Times n_d,A)$.
\end{defi}

So notice that $ex_1(n,A)=\min\{n,|A|-1\}$ and $ex_2(n,A)$ is the usually studied forbidden submatrix problem. 
These notions were previously mainly studied for matrices, see \cite{Gyori,KlazarV,Keszegh,PachTar,Seth,SethDS,SethDisproveFH,Tardos}. 

We have the following monotonicity for matrices with forbidden submatrices, similar to that of the monotonicity of density of graphs with forbidden subgraphs as shown in \cite{TuranDensity}.

\begin{claim}\label{exmon} If for $1\le i\le d$ we have $m_i\le n_i$, then
$$ex_d(n_1\Times  n_d,A)\le
\frac{n_1}{m_1}\cdots \frac{n_d}{m_d} ex_d(m_1\Times  m_d,A).$$
\end{claim}
\begin{proof}
Let $M$ be an $A$-free $d$-matrix of size $n_1\Times  n_d$.
Any $M'$ $d$-submatrix of $M$ is also $A$-free.
If $M'$ is of size $m_1\Times  m_d$, then $|M'|\le ex_d(m_1\Times  m_d,A)$.
Averaging over all submatrices of this size the statement follows, as any entry of $M$ has probability $\frac{m_1}{n_1}\cdots \frac{m_d}{n_d}$ to be in a submatrix.
\end{proof}

From this we can get a bound on $ex_d(n_1\Times  n_d,A)$ from $ex_d(n,A)$.

\begin{cor}\label{square2rectangle}
If for all $n$ we have $ex_d(n,A)\le Kn^{d-1}$, then for all $n_1,\ldots,n_d$ we have $ex_d(n_1\Times  n_d,A)\le K\frac{n_1\cdots n_d}{\min_i n_i}.$
\end{cor}
\begin{proof}
Apply Claim \ref{exmon} for $m_1=\ldots=m_d=\min n_i$.
\end{proof}

We also need to generalize the notion of a permutation matrix to higher dimensions.

\begin{defi}
A $d$-matrix $M$ of size $k^d$ is a {\em permutation $d$-matrix} if $|M|=k$ and
it contains exactly one $1$ in each axis-parallel hyperplane.
In other words, $\forall$ $1\le i_j\le k$ there is a unique $\ul i=(i_1,\ldots,i_d)$
such that $M(\ul i)=1$.
\end{defi}

The most important result about excluded permutation matrices is the theorem of Marcus-Tardos, which was conjectured by F\"uredi and Hajnal \cite{FurHaj} and shown by Klazar \cite{Klazar} to also imply the Stanley-Wilf conjecture.

\begin{thm}[Marcus-Tardos \cite{MarTar}]\label{marcustardos} 
If $A$ is a permutation matrix of size $k^2$, then $ex_2(n,A)\le 2k^4 \binom {k^2}{k} n=O(n)$.
\end{thm}

We need the following straightforward generalization of this result.

\medskip
\noindent{\bf Theorem \ref{hypermarcustardos}} (Klazar-Marcus \cite{KlaMar}){\bf.} {\em
If $A$ is a permutation $d$-matrix of size $k^d$, then $ex_d(n,A) =O(n^{d-1})$.}
\medskip

An equivalent reformulation of Definition \ref{posetdim} gives the following connection between permutation $d$-matrices and posets of dimension $d$.

\begin{obs}\label{poset2matrix} For every poset $P$ of size $k$ whose dimension is $d$, there is a permutation $d$-matrix $M_P$ of size $k^d$ whose $1$-entries are in bijection with the elements of the poset such that if $M_P(\ul i)$ is in bijection with $p\in P$ and 
$M_P(\ul i')$ is in bijection with $p'\in P$, then
$p<p' \Leftrightarrow \forall j\:\: i_j< i_j'.$ 
\end{obs}

Using this, we will prove our main lemma.

\medskip
\noindent{\bf Lemma \ref{matrix2poset}.} {\em
If $ex_d(n,M_P)\le Kn^{d-1}$,
then $La^{\#}(n, P) \le 2^dK {\binom{n}{\lfloor\frac{n}{2}\rfloor}}$.}
\medskip

\section{Proof of Lemma \ref{matrix2poset}}\label{sec:matrix2poset}
Before we start the proof, we need a definition.

\begin{defi}
A {\em permutation $d$-partition} $\Q := Q_1 | Q_2 | \ldots | Q_{d}$ is defined as an ordered partition of a permutation of the elements of $[n]$ into $d$ parts $Q_1, Q_2,\ldots, Q_d$. 
We denote the $i^{th}$ element of $Q_j$ by $Q_j(i)$.
The set of the form $Q_j[i):=\{Q_j(1),\ldots,Q_j(i-1)\}$ is called a {\em prefix} of $Q_j$ and
the set of the form $\Q[\ul i):=\bigcup_{j=1}^d Q_j[i_j)$ is called a {\em prefix union} of $\Q$. 
\end{defi}

An example of a {permutation $3$-partition} is $\Q=142|5|3$ and $\Q[3,1,2)=\{1,3,4\}$ is a prefix union of $142|5|3$.
Notice that since the order of the parts is respected, we consider, say, $\Q'=5|142|3$ as a different permutation $3$-partition, but of course the prefixes of $\Q$ and $\Q'$ are the same.
Some parts might also be empty in a permutation $3$-partition, as in $142||53$. 

The total number of possible permutation $d$-partitions is easily seen to be $(n+d-1)!/(d-1)!$, by taking all permutations of the elements of $[n]$ and the ${d-1}$ separators.\\ 

We are now ready to start the proof of Lemma \ref{matrix2poset}.
Let $\F$ be an induced $P$-free family of subsets of $[n]$.
We double count pairs $(\Q,F)$ where $F \in \F$ and $F$ is a prefix union of $\Q$.

First, let us fix a set $F \in \F$ and calculate the number of permutation $d$-partitions $\Q$ such that $F$ is a prefix union of $\Q$.

\begin{claim}\label{countP} Let $F \in \F$. Then,

\begin{equation*}
|\{ \Q\mid F \textit{ is a prefix union of } \Q \}|=
\frac{(\abs{F} +d-1)!}{(d-1)!} \cdot \frac{(n- \abs{F} +d-1)!}{(d-1)!} \\ =
\frac{(n+2d-2)!}{((d-1)!)^2\binom{n+2d-2}{\abs{F} +d-1}}.
\end{equation*}
\end{claim}
\begin{proof}
Permute the elements of $F$ and $d-1$ separators $``|"$ in $(\abs{F}+d-1)!/(d-1)!$ ways.
Each such permutation is of the form $L_1|L_2|\ldots|L_{d}$.
Also permute the elements of $[n] \setminus F$ and $d-1$ separators $``|"$ in $(n-\abs{F}+d-1)!/(d-1)!$ ways.
Each such permutation is of the form $R_1|R_2|\ldots|R_{d}$.

Now, we concatenate $L_1|L_2|\ldots|L_{d}$ and $R_1|R_2|\ldots|R_{d}$ as $L_1R_1|L_2R_2|\ldots|L_{d}R_{d}$ to obtain a permutation $d$-partition for which $F$ is a prefix union. 
\end{proof}

Now, let us fix a $\Q = Q_1 | Q_2 | \ldots | Q_{d}$ and calculate the number of sets $F \in \F$ such that $F$ is a prefix union of $\Q$.
To do this, we first associate a $d$-matrix $M_{\Q}$ of size $(\abs{Q_1} + 1)\Times  (\abs{Q_d} + 1)$ to $\Q$ where $\abs{Q_j}$ denotes the length of $Q_j$.
This is done by setting $M_{\Q}(\ul i)=1$ if
the prefix union $\Q[\ul i) \in \F$
and $M_{\Q}(\ul i)=0$ otherwise.

\begin{claim}\label{countA}
If $ex_d(n,M_P)\le Kn^{d-1}$ and $\F$ is induced $P$-free,
then for any $\Q$ we have
$$|\{F\in\F\mid F \textit{ is a prefix union of } \Q\}|
\le K\left(\frac{n+d-1}{d-1}\right)^{d-1}=O(n^{d-1}).$$
\end{claim}
\begin{proof}
Consider the permutation $d$-matrix $M_P$ of size $|P|^d$ defined in Observation \ref{poset2matrix}.
Notice that %if $\forall \:  1 \le j \le d : 1\le i_j, i_j' \le \abs{Q_j}+1$, then
$\Q[\ul i')\subset \Q[\ul i)$ if and only if $\forall \:  1 \le j \le d : i_j'\le i_j$
and equality can hold only if $\ul i=\ul i'$.
From this it follows that if $M_{\Q}$ contains $M_P$, then the same relations hold in $\F$ and thus $\F$ contains an induced $P$, which is impossible.
Therefore $M_{\Q}$ is $M_P$-free. Using Corollary \ref{square2rectangle} and $ex_d(n,M_P)\le Kn^{d-1}$, we have 

\begin{multline*}
|\{F\in\F\mid F \textit{ is a prefix union of } \Q\}|=|M_{\Q}|
\le K\frac{(\abs{Q_1} + 1)\cdots (\abs{Q_d} + 1)}{\min_j (\abs{Q_j} + 1)}
\le K\left(\frac{n+d-1}{d-1}\right)^{d-1}.
\end{multline*}
\end{proof}

Now using Claims \ref{countP} and \ref{countA}, we have

\begin{multline*}
\sum_{F \in \F} {\frac{(n+2d-2)!}{((d-1)!)^2\binom{n+2d-2}{\abs{F}+d-1}}}\le
|\{(\Q,F)\mid F\in\F \textit{ prefix union of } \Q\}|\le
\frac{(n+d-1)!}{(d-1)!} K \left(\frac{n+d-1}{d-1}\right)^{d-1}
\end{multline*}

simplifying which, we get,

\begin{equation}
\frac{\abs{\F}}{\binom{n+2d-2}{\lfloor \frac{n}{2} \rfloor + d-1}} \le \sum_{F \in \F} {\frac{1}{\binom{n+2d-2}{\abs{F}+d-1}}}\le
K \frac{(d-1)!}{(d-1)^{d-1}}
\end{equation}

and since $\binom{n+2d-2}{\lfloor \frac{n}{2} \rfloor + d-1}\le 4^{d-1}\binom{n}{\lfloor \frac{n}{2} \rfloor }$, using Stirling's formula Lemma $2$ follows. 
In fact, we actually get the following stronger result:
$La^{\#}(n, P) = O(1.48^d  K{\binom{n}{\lfloor\frac{n}{2}\rfloor}})$.

\section{Proof of Theorem \ref{hypermarcustardos}}\label{sec:hypermarcustardos}
The proof is similar to the proof of Marcus and Tardos \cite{MarTar}, except that we use induction on $d$, just like Klazar and Marcus \cite{KlaMar}.
Surprisingly, even though both the Klazar-Marcus proof and our proof are a very natural generalization of the Marcus-Tardos proof, they are still quite different. Below we present our proof.

We assume by induction that for all smaller $d$ the following statement is true: Any $d$-matrix without some permutation $d$-matrix of size $k$ has at most $C_d n^{d-1}$ elements where $C_d$ denotes the smallest possible constant.
The theorem trivially holds for $d=1$.

Let $M$ be a $d$-matrix of size $n^d$ and $A$ a permutation $d$-matrix of size $k^d$.
If $S$ is a $d$-matrix, denote by $Proj_i S$ the $(d-1)$-matrix obtained by orthogonally projecting $S$ to the hyperplane orthogonal to the $i^{th}$ axis.
Notice that $Proj_i A$ is a permutation $(d-1)$-matrix of size $k^{d-1}$.

We partition $M$ into smaller $d$-matrices of size $s^d$ called {\em blocks} (for simplicity, suppose that $n$ is divisible by $s$) such that for every $1\le b_1,\ldots,b_d\le n/s$ the elements of block $S_{\ul b}$ are $M(\ul a)$ for $sb_i-s< a_i\le sb_i$. 
An {\em $i$-blockcolumn} is a series of blocks parallel to the $i^{th}$-axis, i.e., $\{S_{\ul b}\mid b_i=1,\ldots,n/s\}$.

A block $S$ is called {\em $i$-wide} if $Proj_i S$ contains $Proj_i A$ as a $(d-1)$-submatrix.
Using induction on the dimension, if this is not the case, $|Proj_i S| \le C_{d-1} s^{d-2}  =O(s^{d-2})$. 
If a block is not wide for any $i=1,\ldots,d$, we call it {\em thin}.

For the induction, we also need to use the following well-known inequality,
usually credited to Loomis and Whitney, which is in fact a simple corollary of the submodularity of entropy \cite{mathSE}.

\begin{lem}[Loomis-Whitney \cite{LooWhi}]
$|S|^{d-1}\le \Pi_{i=1}^d |Proj_i S|$.
\end{lem}

If a block $S$ is thin, then using the above inequality, by induction we get
$|S|\le {(C_{d-1}s^{d-2})}^{\frac{d}{d-1}} = o(s^{d-1})$.
The number of $i$-wide blocks in an $i$-blockcolumn is at most $(k-1){\binom{s^{d-1}}{k}}$, because if $Proj_i A$ would occur $k$ times, in the same $k$ $i$-columns, then we could ``build'' a copy of $A$ from them (here we use that $A$ is a permutation $d$-matrix).

We define the $d$-matrix $M'$ of size $(n/s)^d$ as $M_{\underline i}'=1$ if and only if block 
$S_{\ul b}$ is thin. 
As $M'$ must be also $A$-free, we get the following bound by induction on $d$ and $n$, where $k$ is fixed.

\begin{multline}%\limits
|M|\le \sum_{S \textit{ is thin}} |S| +
\sum_{i=1}^d \sum_{\substack{BC\textit{ is an}\\ i \textit{-blockcolumn}}}
\sum_{\substack{S\in BC \textit{ is }\\ i \textit{-wide}}} |S|
\le \sum_{S \textit{ is thin}} o(s^{d-1}) +
\sum_{i=1}^d \sum_{\substack{BC\textit{ is an}\\ i \textit{-blockcolumn}}}
\sum_{\substack{S\in BC \textit{ is }\\ i \textit{-wide}}} s^d\\
\le |M'| o(s^{d-1}) +
\sum_{i=1}^d \sum_{\substack{BC\textit{ is an}\\ i \textit{-blockcolumn}}}
 (k-1){s^{d-1}\choose k} s^d
\le C_d(n/s)^{d-1}o(s^{d-1})+d(n/s)^{d-1} (k-1){s^{d-1}\choose k}s^d
\end{multline}

which for a sufficiently large $s$ is less than $(1-\delta)C_dn^{d-1} + s^{dk}n^{d-1}\le C_dn^{d-1}$. %\medskip
%
%
%We can bound $C_d$ by $k^{k^d((d+1)!)^2}=2^{k^{\Theta(d)}}$ using induction and the previous inequality.
%
%
%$$
%|M| \le C_d(n/s)^{d-1} C_{d-1}^{\frac{d}{d-1}} s^{\frac {d(d-2)}{d-1}}+d(n/s)^{d-1} (k-1){s^{d-1}\choose k}s^d
%\le C_d k^{k^{d-1}(d!)^2{\frac{d}{d-1}}} s^{\frac {-1}{d-1}}n^{d-1}+s^{dk}n^{d-1}
%$$
%
%which is at most $C_dn^{d-1}$ if and only if
%$C_d\ge \frac{s^{dk}}{1-k^{k^{d-1}(d!)^2{\frac{d}{d-1}}} s^{\frac {-1}{d-1}}}$.
%If we select $s=k^{k^{d-1}(d!)^2(d+1)}$, then the right-hand side becomes
%
%$$\frac{k^{k^{d}(d!)^2d(d+1)}}{1 - k^{k^{d-1}(d!)^2{\frac{-1}{d-1}}}} 
%\le 2k^{k^{d}(d!)^2d(d+1)}
%\le k^{k^d((d+1)!)^2},$$ 
%
%so $C_d\le k^{k^d((d+1)!)^2}=2^{k^{\Theta(d)}}$.

\section{Concluding Remarks and the Diamond}\label{sec:concluding}
Let $P$ be a poset.
A $0-1$ submatrix is a {\em $P$-pattern} if the usual partial order among the entries of the submatrix induces $P$.
E.g., $M_P$ is a $P$-pattern.
Notice that in the proof of Lemma \ref{matrix2poset}, we only used that $M_{\Q}$ is $M_P$-free, when in fact, $M_{\Q}$ avoids all $P$-patterns.
This might help to improve constant term of Theorem \ref{mainthm} for specific posets, as Theorem \ref{hypermarcustardos} gives a very weak bound (although there is hope that the bound of that theorem can be improved as well).
%Although obtaining general bounds by avoiding all $P$-patterns seems to be difficult, many examples suggest that there is hope for a significant improvement in our current constant term of Theorem \ref{mainthm} for specific posets.

For example, denote by $D_2$ the {\em diamond} poset, the poset on $4$ elements whose elements are $a<b,c<d$ where $b$ and $c$ are incomparable.
G\'abor Tardos \cite{Diamond} proved that if all $16$ $D_2$-patterns are avoided, then an $n\times n$ matrix can have at most $4n$ one entries in it. %, where $D_2$ is the diamond poset.
Using this bound, and the analogue of Lemma \ref{matrix2poset} for all patterns, we have
$La^{\#}(n, D_2) \le 16 \binom{n}{\lfloor \frac{n}{2} \rfloor}$. 
If instead of $M_{\Q}$ we build a matrix in a slightly different way, for the diamond, we could improve this to $La^{\#}(n, D_2) \le 6 \binom{n}{\lfloor \frac{n}{2} \rfloor}$. 
However, this is still very far from the recent result of Lu and Milans \cite{Milans} which gives
$La^{\#}(n, D_2) \le (2.583+o(1)) \binom{n}{\lfloor \frac{n}{2} \rfloor}$.
For the non-induced case Kramer, Martin and Young \cite{KMY} proved the slightly better bound
$La(n, D_2) \le (2.25+o(1)) \binom{n}{\lfloor \frac{n}{2} \rfloor}$.
The conjecture is $La^{\#}(n, D_2) \le (2+ o(1)) \binom{n}{\lfloor \frac{n}{2} \rfloor}$.
%(SEE http://arxiv.org/pdf/1309.5638.pdf FOR NICE CONSTRUCTIONS CLOSER TO UPPERBOUND)
%
%We wonder if using some other matrix instead of $M_{\Q}$, can improve our bounds to get sharp results.
We hope that using some different matrix instead of $M_{\Q}$, might improve our bounds to give sharp results.

\medskip

Lu and Milans have obtained bounds on $La^{\#}(n,P)$ by considering the {\em Lubell function} of $\F$ defined as $\sum_{F \in \F} {\frac{1}{\binom{n}{|F|}}}$.
They proposed the following strengthening of Conjecture \ref{conj}.

\begin{conj}[Lu-Milans \cite{Milans}]
\label{conj2}
For every poset $P$, there is a constant $C$ such that for every induced $P$-free $\F$ we have
$ \sum_{F \in \F} {\frac{1}{\binom{n}{|F|}}} \le C.$
\end{conj}

Unfortunately we could not establish this conjecture, as our results only imply that \\ $\sum_{F \in \F} {\frac{1}{\binom{n+2d-2}{\abs{F}+d-1}}} = O(1)$ because the terms of this sum corresponding to sets $F$ which are in the region $\abs{F} \le \alpha$ and $\abs{F} \ge n - \alpha$, for any fixed constant number $\alpha > 0$, can be arbitrarily smaller than the corresponding terms in the Lubell function. It is, however, quite interesting that the induced diamond-free family $\F$ for which Lu and Milans \cite{Milans} show that $\sum_{F \in \F} {\frac{1}{\binom{n}{|F|}}} \ge 2.28$, has only sets of size at most $3$. This construction is similar to the one given earlier by Griggs, Li and Lu \cite{GLL}, which gives a non-induced diamond-free family $\F$ for which $\sum_{F \in \F} {\frac{1}{\binom{n}{|F|}}} \ge 2.25$ and only uses sets of size at most $2$.
This prompts us to propose the following conjecture.

\begin{conj}
For every induced diamond-free $\F$, we have
$\sum_{F \in \F} {\frac{1}{\binom{n+2}{\abs{F}+1}}}\le \frac{1}{2}+o(1)$.
\end{conj}
%THIS IS A STRENGTHENING OF THE DIAMOND CONJECTURE \ref{DiamondConjecture}.

In general, if we had a better matrix instead of $M_{\Q}$ that we could use, we could even hope for the following daring conjecture about our ``shifted'' Lubell function, which generalizes a conjecture of Griggs and Lu \cite{GLu}.
Let $e(P)$ (resp.\ $e^{\#}(P)$) be the maximum number $m$ such that for all $n$, the union of the $m$ middle levels of $2^{[n]}$ is $P$-free (resp.\ induced $P$-free).

\begin{conj}
If $P$ is a poset of dimension $d$, then for every $P$-free family $\F$ we have\newline
$\sum_{F \in \F} {\frac{1}{\binom{n+2d-2}{\abs{F}+d-1}}}\le \frac{e(P)}{4^{d-1}} +o(1)$ and thus $|\F|\le (e(P) + o(1))\binom{n}{\lfloor n/2\rfloor}$
and for every induced $P$-free $\F$ we have
$\sum_{F \in \F} {\frac{1}{\binom{n+2d-2}{\abs{F}+d-1}}}\le \frac{e^{\#}(P)}{4^{d-1}} +o(1)$ and thus $|\F|\le (e^{\#}(P) + o(1))\binom{n}{\lfloor n/2\rfloor}$.
\end{conj}

\section{Acknowledgements}
We thank Gyula O.H. Katona for introducing forbidden subposet problems to us and G\'abor Tardos for calculating the upperbound for Diamond patterns, almost overnight.
We would also like to thank Kevin Milans and Bal\'azs Patk\'os for useful discussions and Jan Kyncl for reading our manuscript and pointing out that Theorem \ref{hypermarcustardos} has been already proved by Klazar and Marcus in \cite{KlaMar}.

%THIS SHOULD COME IN YOUR MASTER'S THESIS OR PHD, BUT I'VE NEVER SEEN IT IN SIMPLE PAPERS. 
%We would also like to thank Padmini Mukkamala and Anna Jablonkai for their support.  

\end{document}